\documentclass[twoside,11pt,reqno]{amsart}

\usepackage{amsmath,amsthm,amssymb,amstext,amsfonts,amscd}
\usepackage{graphicx}
\usepackage{multirow}
\usepackage{url}

\newtheorem{theorem}{Theorem}

\numberwithin{equation}{section}

\begin{document}
\title[{Summation formulas for the Kamp\'e de F\'eriet function}]
{\bf  Further summation formulas for the Kamp\'e de F\'eriet function}
\author[{\bf J. Choi, Arjun K. Rathie}]{\bf Junesang Choi$^*$ and Arjun K. Rathie}
\address{Junesang Choi: Department of Mathematics, Dongguk University, Gyeongju 38066,  Republic of Korea}
\email{junesang@dongguk.ac.kr}
\bigskip
\address{Arjun K. Rathie :Department of Mathematics, Vedant College of Engineering $\&$ Technology, Rajasthan Technical University, Village: TULSI, Post: Jakhmund, District: BUNDI-323021, Rajasthan State, India}
  \email{arjunkumarrathie@gmail.com}

\thanks{$^*$ Corresponding author}



\begin{abstract}

The aim of this research is to provide thirty-two interesting summation formulas for the Kamp\'e de F\'eriet function
in general forms, which are given in sixteen theorems.
The results are established with the help of the identities in
Liu and Wang \cite{Li-Wa} and  generalizations of Kummer's summation theorem, Gauss' second summation theorem and Bailey's summation theorem obtained earlier  by Rakha and Rathie \cite{Ra-Ra}.
Some special cases and relevant connections of the results presented here with those involving certain known identities are also indicated.

\end{abstract}

\maketitle

\section{Introduction and Preliminaries} \label{sec-1}

  The natural generalization of the Gauss's hypergeometric function ${}_2F_1$ is called
  the  generalized hypergeometric series $_pF_q$ $\left(p,\, q \in \mathbb{N}_0\right)$ defined by (see, e.g.,  \cite{Bail-64}, \cite[p.~73]{Rain} and
  \cite[pp. 71-75]{Sr-Ch-12}):
\begin{equation}\label{pFq}  
 \aligned
 {}_pF_q \left[ \aligned \alpha_1,\,\ldots,\,\alpha_p &;\\
                       \beta_1,\,\ldots,\,\beta_q &; \endaligned
                  \,\, z \right] =& \sum_{n=0}^\infty \, {(\alpha_1)_n \cdots (\alpha_p)_n
               \over (\beta_1)_n \cdots (\beta_q)_n} {z^n \over n!} \\
               =&\,\, _pF_q (\alpha_1,\,\ldots,\,\alpha_p ;\, \beta_1,\,\ldots,\,\beta_q;\,z),
\endaligned
\end{equation}
where $(\lambda)_n$ is the Pochhammer symbol defined (for $\lambda \in \mathbb C$) by (see \cite[p.~2 and p.~5]{Sr-Ch-12}):
\begin{equation}\label{Pochhammer symbol}  
\aligned (\lambda)_n :
 & = \frac{\Gamma (\lambda +n)}{\Gamma(\lambda)}
      \quad (\lambda \in {\mathbb C} \setminus {\mathbb Z}_0^-)\\
 & = \left\{\aligned & 1  \hskip 45 mm (n=0) \\
        & \lambda (\lambda +1) \ldots (\lambda+n-1) \hskip 8mm  (n \in {\mathbb N})
   \endaligned \right. \\
\endaligned
\end{equation}
 and $\Gamma (\lambda)$ is the familiar Gamma function.
 Here  an empty product is interpreted as $1$, and we assume (for simplicity) that the variable
$z,$ the numerator parameters
 $\alpha_1,$ $\ldots,$ $\alpha_p,$ and the denominator   parameters
 $\beta_1,$ $\ldots,$ $\beta_q$ take on complex values, provided that no zeros appear in the
 denominator of \eqref{pFq}, that is,
  \begin{equation}\label{Beta-j}  
  (\beta_j \in {\mathbb C} \setminus {\mathbb Z}_0^-; \,\, j=1,\,\ldots, q).
\end{equation}
Here and in the following, let $\mathbb{C}$, $\mathbb{Z}$ and $\mathbb{N}$ be
the sets of complex numbers, integers  and positive integers, respectively,
 and let
 $$ \mathbb{N}_0:=\mathbb{N} \cup \{0\} \quad \text{and} \quad \mathbb{Z}_0^-:= \mathbb{Z}\setminus \mathbb{N}.$$

For more details of ${}_pF_q$ including its convergence, its various special and limiting cases,
and its further diverse generalizations,
 one may be  referred, for example, to \cite{Bail-64, Er-Ma-Ob-Tr-I, Exto-76, Rain,   Slat-66, Sr-Ch-12,  Sr-Ka, Sr-Ma}.

It is worthy of note that whenever the generalized hypergeometirc function ${}_pF_q$
(including ${}_2F_1$) with its specified argument (for example, unit argument or $\frac{1}{2}$ argument)
can be summed to be expressed in terms of the Gamma functions,
the result may be very important from both theoretical and applicable points of view.
Here, the classical summation theorems for the generalized hypergeometric series
such as those of Gauss and Gauss second, Kummer, and Bailey for the series ${}_2F_1$;
Watson's, Dixon's, Whipple's and Saalsch\"utz's summation theorems for the series
${}_3F_2$ and others play  important roles in theory and application.
During $1992$-$1996$, in a series of works,
Lavoie et al. \cite{La-Gr-Ra-92, La-Gr-Ra-Ar, La-Gr-Ra} have generalized the above mentioned classical summation theorems for ${}_3F_2$
of Watson, Dixon, and Whipple and presented a large number of special and limiting cases of their results, which have been further generalized
and extended by Rakha and Rathie \cite{Ra-Ra} and Kim et al. \cite{Ki-Ra-Ra}.
Those results have also been obtained and verified with the help of computer programs (for example,
Mathematica).

The vast popularity and immense usefulness of the hypergeometric function and the generalized hypergeometric functions of one
variable have inspired and stimulated a large number of researchers to introduce and investigate hypergeometric functions of
two or more variables. A serious, significant and systematic study of the hypergeometric functions of two variables
was initiated by Appell \cite{Ap-Ka} who presented the so-called Appell functions $F_1$, $F_2$,  $F_3$ and $F_4$
which are generalizations of the Gauss' hypergeometric function. Here we recall the  Appell function $F_3$
(see, e.g., \cite[p. 23, Eq. (4)]{Sr-Ka})
\begin{equation}\label{Appell-F3}
\aligned  F_3[a,\,a',\,b,\,b';\,c;\,x,y]
    &= \sum_{m,\,n=0}^{\infty}\, \frac{(a)_m\,(a')_n\,(b)_m\,(b')_n}{(c)_{m+n}}\,\frac{x^m}{m!}\,\frac{y^n}{n!} \\
    &= \sum_{m=0}^{\infty}\,\frac{(a)_m\,(b)_m}{(c)_m}\, {}_2F_1 \left[ \aligned a',\, b' &;\\
                       c+m &; \endaligned
                  \,\, y \right]\,\frac{x^m}{m!}
\endaligned
\end{equation}
\begin{equation*}
  (\max \{|x|,\,|y|\}<1).
\end{equation*}
The confluent forms of the Appell functions were studied by
Humbert \cite{Humb-21}. A complete list of these functions can be seen in the standard literature, see, e.g.,
\cite{Er-Ma-Ob-Tr-I}.
Later, the four Appell functions and their confluent forms were further generalized by Kamp\'e de F\'eriet \cite{Kamp}
who introduced more general hypergeometric functions of two variables. The notation defined and introduced by
Kamp\'e de F\'eriet for his double hypergeometric functions of superior order was subsequently abbreviated by
Burchnall and Chaudndy \cite{Bu-Ch-40, Bu-Ch-41}. We recall here the definition of a more general double hypergeometric
function (than one defined by Kamp\'e de F\'eriet) in a slightly modified notation given by Srivastava and Panda
\cite[p. 423, Eq. (26)]{Sr-Pa}. The convenient generalization of the Kamp\'e de F\'eriet function is defined as follows:
\begin{equation}\label{G-KdF-ft}
  \aligned
   & F_{\,G:C;D}^{H:A;B} \left[\begin{array}{ccc}
                               \left(h_H\right): & \left(a_A\right); & \left(b_B\right); \\
                              \left(g_G\right): & \left(c_C\right); & \left(d_D\right);
                             \end{array} x,\,y
   \right] \\
  &\hskip 10mm = \sum_{m=0}^{\infty}\sum_{n=0}^{\infty}\,\frac{\left( \left(h_H\right)\right)_{m+n}\,\left( \left(a_A\right)\right)_{m}\, \left( \left(b_B\right)\right)_{n} }{\left( \left(g_G\right)\right)_{m+n}\,\left( \left(c_C\right)\right)_{m}\, \left( \left(d_D\right)\right)_{n}}\, \frac{x^m}{m!}\,\frac{y^n}{n!},
  \endaligned
\end{equation}
where $\left(h_H\right)$ denotes the sequence of parameters $\left(h_1,h_2,\ldots, h_H\right)$ and $\left( \left(h_H\right)\right)_{n}$ is defined by
the following product of Pochhammer symbols
   \begin{equation*}
     \left( \left(h_H\right)\right)_{n}:= \left(h_1\right)_n \left(h_2\right)_n \cdots \left(h_H\right)_n \quad \left(n \in \mathbb{N}_0\right),
   \end{equation*}
where, when $n=0$, the product is  to be interpreted as unity. For more details about the function \eqref{G-KdF-ft} including its convergence,
 we refer, for example, to \cite{Sr-Ka}.

When some extensively generalized special functions like \eqref{G-KdF-ft} were appeared, it has been an interesting and natural research subject to consider
certain reducibilities of the functions. In this regard, the Kamp\'e de F\'eriet function has attracted many mathematicians to investigate its reducibility
and transformation formulas. In fact, there are numerous reduction formulas and transformation formulas of the Kamp\'e de F\'eriet function in the literature,
see, e.g., \cite{Bu-Sr, Carl-67, Ch-Ch-Ch-Sr, Chen-JMAA, Ch-Sr-JCAM, Ch-Ra-HMJ,Ch-Ra-AMS,Cv-Mi, Jain, Karl, Kim-HMJ, Kurp, Mill, Ra-Aw-Ra, Ra-Je,Sara-80,Sh-Sa,Shar,Shar-76,Sing,Sriv-77,Sr-Da,Van,Va-Pi-Ra,Va-Pi-Ra-,Va-Pi-Ra-JCAM}. In the above-cited references, most of the reduction formulas were related to the cases $H+A=3$ and $G+C=2$. In $2010$,
by using Euler's transformation formula for ${}_2F_1$,  Cvijovi\`c and  Miller \cite{Cv-Mi} established a reduction formula for the case
  $H+A=2$ and $G+C=1$. Motivated essentially by the work \cite{Cv-Mi}, recently, Liu and Wang \cite{Li-Wa} used Euler's first and second transformation formulas
  for ${}_2F_1$ and the above-mentioned classical summation theorems for ${}_pF_q$ to present a number of very interesting reduction
  formulas and then deduced summation formulas for the Kamp\'e de F\'eriet function. Indeed, only a few summation formulas for  the Kamp\'e de F\'eriet function
  are available in the literature.

  In this sequel, we aim to establish $32$ interesting general summation formulas for the Kamp\'e de F\'eriet function in the form of
  $16$ theorems based on the transform formulas obtained recently by Liu and Wang \cite{Li-Wa}.
   We also  demonstrate how easily one can obtain as many as $161$ interesting summation formulas for the Kamp\'e de F\'eriet function, which contain $16$ known formulas.
   The results are established with the help of generalizations of Kummer summation theorem, Gauss second summation theorem and Bailey summation theorem   due to Rakha and Rathie \cite{Ra-Ra}.

\section{Results required} \label{sec-2}

In order to make this paper self-contained, we recall the deduction formulas for the Kamp\'e de F\'eriet function established by Liu and Wang \cite{Li-Wa}.

\vskip 3mm

\begin{equation}\label{eq2-1}
\aligned
 & F_{\,1:0;1}^{1:1;2} \left[\begin{array}{rrr}
                               \alpha: & \epsilon \,;\,& \beta-\epsilon,\, \gamma \,; \\
                             \beta: &\overline{\hspace{3mm}} \,;\,& \gamma+\beta \,;
                             \end{array} x,\,x
   \right] \\
  & \hskip 10mm =(1-x)^{\beta-\epsilon-\alpha}\,{}_2F_1 \left[ \beta-\epsilon,\,\gamma+\beta-\alpha\,;
                       \gamma+\beta \,;
                  \,\,x \right];
\endaligned
\end{equation}

\begin{equation}\label{eq2-2}
\aligned
 & F_{\,1:0;1}^{1:1;2} \left[\begin{array}{rrr}
                               \alpha: & \epsilon \,;\,& \beta-\epsilon,\, \frac{1}{2} \alpha +1 \,; \\
                             \beta: &\overline{\hspace{3mm}} \,;\,& \frac{1}{2} \alpha \,;
                             \end{array} x,\,x
   \right] \\
  & \hskip 10mm =(1-x)^{\beta-\epsilon-\alpha}\,{}_2F_1 \left[ \beta-\epsilon,\,1+\frac{1}{2} \beta \,;\,
                      \frac{1}{2} \beta \,;  \,\,x \right];
\endaligned
\end{equation}

\begin{equation}\label{eq2-3}
\aligned
 & F_{\,1:0;2}^{1:1;3} \left[\begin{array}{rrr}
                               \alpha: & \epsilon \,;\,& \beta-\epsilon,\, 1+\frac{1}{2} \alpha,\, \frac{\alpha-\beta}{2} \,; \\
                             \beta: &\overline{\hspace{3mm}} \,;\,& \frac{1}{2} \alpha,\, 1+ \frac{\alpha+\beta}{2} \,;
                             \end{array} x,\,x
   \right] \\
  & \hskip 10mm =(1-x)^{\beta-\epsilon-\alpha}\,{}_2F_1 \left[ \beta-\epsilon,\,   \frac{\beta-\alpha}{2} \,;\,
                   1+ \frac{\alpha+\beta}{2} \,;  \,\,x \right];
\endaligned
\end{equation}

   \begin{equation}\label{eq2-4}
\aligned
 & F^{0:2;2}_{\,1:0;0} \left[\begin{array}{rrr}
                               \overline{\hspace{3mm}}: &\alpha,\, \epsilon \,;\,& \beta-\epsilon,\,\gamma \,; \\
                             \beta: &\overline{\hspace{3mm}} \,;\,& \overline{\hspace{3mm}}  \,;
                             \end{array} x,\,\frac{x}{x-1}
   \right] \\
    & \hskip 10mm = F_3 \left(\alpha,\, \beta-\epsilon\,:\,\epsilon,\,\gamma\,;\,\beta\,;\, x,\, \frac{x}{x-1} \right)\\
   & \hskip 10mm =(1-x)^{-\alpha}\,{}_2F_1 \left[ \beta-\epsilon,\,  \alpha+\gamma \,;\,
                   \beta \,;  \,\,\frac{x}{x-1} \right];
\endaligned
\end{equation}

   \begin{equation}\label{eq2-5}
\aligned
 & F^{2:0;1}_{\,1:0;1} \left[\begin{array}{rrr}
                            \alpha,\, \gamma   : &\overline{\hspace{3mm}} \,;\,& \epsilon \,; \\
                             \beta: &\overline{\hspace{3mm}} \,;\,& \beta+\epsilon  \,;
                             \end{array} x,\,-x
   \right] \\
      & \hskip 10mm =(1-x)^{-\alpha}\,{}_2F_1 \left[ \beta-\epsilon,\,  \alpha+\gamma \,;\,
                   \beta \,;  \,\,\frac{x}{x-1} \right];
\endaligned
\end{equation}

     \begin{equation}\label{eq2-6}
\aligned
 & F^{2:0;1}_{\,1:0;1} \left[\begin{array}{rrr}
                            \alpha,\, \gamma   : &\overline{\hspace{3mm}} \,;\,& \frac{1}{2}\gamma +1\,; \\
                             \beta: &\overline{\hspace{3mm}} \,;\,& \frac{1}{2}\gamma  \,;
                             \end{array} x,\,-x
   \right] \\
      & \hskip 10mm =(1-x)^{-\alpha}\,{}_2F_1 \left[ \alpha,\,  1+\frac{1}{2}\beta \,;\,
                   \frac{1}{2}\beta \,;  \,\,\frac{x}{x-1} \right];
\endaligned
\end{equation}

     \begin{equation}\label{eq2-7}
\aligned
 & F^{2:0;2}_{\,1:0;2} \left[\begin{array}{rrr}
                            \alpha,\, \gamma   : &\overline{\hspace{3mm}} \,;\,& 1+\frac{1}{2}\gamma,\, \frac{\gamma-\beta}{2} \,; \\
                             \beta: &\overline{\hspace{3mm}} \,;\,& \frac{1}{2}\gamma,\, 1+\frac{\gamma+\beta}{2} \,;
                             \end{array} x,\,-x
   \right] \\
      & \hskip 10mm =(1-x)^{-\alpha}\,{}_2F_1 \left[ \alpha,\,  \frac{\beta-\gamma}{2} \,;\,
                1 + \frac{\gamma+\beta}{2} \,;  \,\,\frac{x}{x-1} \right].
\endaligned
\end{equation}

\vskip 3mm
In addition, we also recall the following generalizations of Kummer summation theorem, Gauss second summation theorem,
and Bailey's summation theorem (see, e.g., \cite{Ra-Ra}):

\textbf{Generalizations of Kummer's summation theorem}
\begin{equation}\label{G-KummerST-Ra-Ra-1} 
 \aligned
 & {}_2F_1 \left[ \aligned a,\,b&\,;\\
                      1+a-b+i &\,; \endaligned
                  \,\,-1 \right]   =\frac{2^{i-2b}\, \Gamma (b-i)\,\Gamma (1+a-b+i)}{\Gamma(b)\,\Gamma (a-2b+i+1)}\\
 &\hskip 25mm  \times  \sum_{r=0}^{i} \,(-1)^r\,\binom{i}{r}\, \frac{\Gamma \left(\frac{a+r+i+1}{2}-b\right)}{\Gamma \left(\frac{a+r-i+1}{2}\right)}
   \quad   \left(i \in \mathbb{N}_0  \right)
   \endaligned
\end{equation}
and
\begin{equation}\label{G-KummerST-Ra-Ra-2} 
 \aligned
 & {}_2F_1 \left[ \aligned a,\,b&\,;\\
                      1+a-b-i &\,; \endaligned
                  \,\,-1 \right]   =\frac{2^{-i-2b}\,\Gamma (1+a-b-i)}{\Gamma (a-2b-i+1)}\\
 &\hskip 25mm  \times  \sum_{r=0}^{i} \,\binom{i}{r}\, \frac{\Gamma \left(\frac{a+r-i+1}{2}-b\right)}{\Gamma \left(\frac{a+r-i+1}{2}\right)}
   \quad   \left(i \in \mathbb{N}_0  \right).
   \endaligned
\end{equation}

\textbf{Generalizations of Gauss's second summation theorem}
\begin{equation}\label{G-Gauss-SST-Ra-Ra-1} 
 \aligned
 & {}_2F_1 \left[ \aligned a,\,b&\,;\\
                     \frac{1}{2}(a+b+i+1) &\,; \endaligned
                  \,\,\frac{1}{2} \right]   =\frac{2^{b-1}\, \Gamma \left(\frac{a+b+i+1}{2}\right)\,\Gamma \left(\frac{a-b-i+1}{2}\right)}{\Gamma(b)\,\Gamma \left(\frac{a-b+i+1}{2}\right)}\\
 &\hskip 30mm  \times  \sum_{r=0}^{i} \,(-1)^r\,\binom{i}{r}\, \frac{\Gamma \left(\frac{b+r}{2}\right)}{\Gamma \left(\frac{a+r-i+1}{2}\right)}
   \quad   \left(i \in \mathbb{N}_0  \right)
   \endaligned
\end{equation}
and
\begin{equation}\label{G-Gauss-SST-Ra-Ra-2} 
 \aligned
 & {}_2F_1 \left[ \aligned a,\,b&\,;\\
                     \frac{1}{2}(a+b-i+1) &\,; \endaligned
                  \,\,\frac{1}{2} \right]   =\frac{2^{b-1}\, \Gamma \left(\frac{a+b-i+1}{2}\right)}{\Gamma(b)}\\
 &\hskip 20mm  \times  \sum_{r=0}^{i} \,\binom{i}{r}\, \frac{\Gamma \left(\frac{b+r}{2}\right)}{\Gamma \left(\frac{a+r-i+1}{2}\right)}
   \quad   \left(i \in \mathbb{N}_0  \right).
   \endaligned
\end{equation}

\textbf{Generalizations of Bailey's summation theorem}
\begin{equation}\label{G-Bailey-ST-Ra-Ra-1} 
 \aligned
 & {}_2F_1 \left[ \aligned a,\,1-a +i&\,;\\
                     b &\,; \endaligned
                  \,\,\frac{1}{2} \right]   =\frac{2^{i-a}\, \Gamma \left(a-i\right)\,\Gamma \left(b\right)}{\Gamma(a)\,\Gamma \left(b-a\right)}\\
 &\hskip 10mm  \times  \sum_{r=0}^{i} \,(-1)^r\,\binom{i}{r}\, \frac{\Gamma \left(\frac{b-a+r}{2}\right)}{\Gamma \left(\frac{b+a+r}{2}-i\right)}
   \quad   \left(i \in \mathbb{N}_0  \right)
   \endaligned
\end{equation}
and
\begin{equation}\label{G-Bailey-ST-Ra-Ra-2} 
 \aligned
 & {}_2F_1 \left[ \aligned a,\,1-a -i&\,;\\
                     b &\,; \endaligned
                  \,\,\frac{1}{2} \right]\\
                  & \hskip 10mm   =\frac{2^{-i-a}\, \Gamma \left(b\right)}{\Gamma \left(b-a\right)}\,
  \sum_{r=0}^{i} \,\binom{i}{r}\, \frac{\Gamma \left(\frac{b-a+r}{2}\right)}{\Gamma \left(\frac{b+a+r}{2}\right)}
   \quad   \left(i \in \mathbb{N}_0  \right).
   \endaligned
\end{equation}

It is remarked in passing that the results \eqref{G-KummerST-Ra-Ra-1}, \eqref{G-Gauss-SST-Ra-Ra-1} and
\eqref{G-Bailey-ST-Ra-Ra-1} are recorded earlier in \cite{Bryc, Pr-Br-Ma-3}
and the results \eqref{G-KummerST-Ra-Ra-1} to \eqref{G-Bailey-ST-Ra-Ra-2}
 are also recorded in \cite{Br-Ki-Ra}.
Further, if we set $i=0,\,1,\,2,\,3,\,4,\,5$ in \eqref{G-KummerST-Ra-Ra-1} and \eqref{G-KummerST-Ra-Ra-2},
\eqref{G-Gauss-SST-Ra-Ra-1} and \eqref{G-Gauss-SST-Ra-Ra-2},
\eqref{G-Bailey-ST-Ra-Ra-1} and \eqref{G-Bailey-ST-Ra-Ra-2}, we get
  the following summation formulas obtained earlier by
   Lavoie et al. \cite{La-Gr-Ra} in compact forms:
\begin{equation}\label{G-KummerST}
 \aligned
 & {}_2F_1 \left[ \aligned a,\,b&\,;\\
                      1+a-b+i &\,; \endaligned
                  \,\,-1 \right]
 = \frac{2^{-a}\, \Gamma\left(\frac{1}{2}\right)\,\Gamma\left(1+a-b+i\right)\,\Gamma\left(1-b\right) }{\Gamma\left(1-b+\frac{1}{2}i+\frac{1}{2}|i|\right)}\\
 &\hskip 10mm  \times \Bigg\{\frac{\mathcal{A}_i}{ \Gamma\left(\frac{1}{2}a-b+  \frac{1}{2}i+1\right)\,\Gamma\left(\frac{1}{2}a + \frac{1}{2}+\frac{1}{2}i-\left[\frac{i+1}{2}\right]
  \right) } \\
  &\hskip 30mm  + \frac{\mathcal{B}_i}{ \Gamma\left(\frac{1}{2}a-b+  \frac{1}{2}i+\frac{1}{2}\right)\,\Gamma\left(\frac{1}{2}a + \frac{1}{2}i-\left[\frac{i}{2}\right]
  \right) } \Bigg\},
 \endaligned
\end{equation}
for $i=0,\, \pm 1,\, \pm 2,\, \pm 3,\, \pm 4,\, \pm 5$.  The coefficients $ \mathcal{A}_i$ and  $ \mathcal{B}_i$  are given in the
Table \ref{table-Ai-Bi}.
Here and in the following, $[x]$ is the greatest integer less than or equal to $x$ and $|x|$
is the absolute value (modulus) of $x$.

\begin{equation}\label{G-GaussSST}
 \aligned
 & {}_2F_1 \left[ \aligned a,\,b&\,;\\
                      \frac{1}{2}(a+b+i+1) &\,; \endaligned
                  \,\,\frac{1}{2} \right]
 = \frac{\Gamma\left(\frac{1}{2}\right)\,\Gamma\left(\frac{1}{2}a +\frac{1}{2}b+\frac{1}{2}i+\frac{1}{2}\right)\,\Gamma\left(\frac{1}{2}a -\frac{1}{2}b-\frac{1}{2}i+\frac{1}{2}\right) }{\Gamma\left(\frac{1}{2}a -\frac{1}{2}b+\frac{1}{2}+\frac{1}{2}|i|\right)}\\
 &\hskip 10mm  \times \Bigg\{\frac{\mathcal{C}_i}{ \Gamma\left(\frac{1}{2}a+  \frac{1}{2}\right)\,\Gamma\left(\frac{1}{2}b + \frac{1}{2}i+\frac{1}{2}-\left[\frac{i+1}{2}\right]
  \right) } + \frac{\mathcal{D}_i}{ \Gamma\left(\frac{1}{2}a\right)\,\Gamma\left(\frac{1}{2}b + \frac{1}{2}i -\left[\frac{i}{2}\right]
  \right) } \Bigg\},
 \endaligned
\end{equation}
for $i=0,\, \pm 1,\, \pm 2,\, \pm 3,\, \pm 4,\, \pm 5$.  The coefficients $ \mathcal{C}_i$ and  $ \mathcal{D}_i$  are given in the
Table \ref{table-Ci-Di}.

\begin{equation}\label{G-BaileySST}
 \aligned
 & {}_2F_1 \left[ \aligned a,\,1-a+i&\,;\\
                      b &\,; \endaligned
                  \,\,\frac{1}{2} \right]
 = 2^{1+i-b}\frac{\Gamma\left(\frac{1}{2}\right)\,\Gamma\left(b\right)\,\Gamma\left(1-a\right) }{\Gamma\left(1-a+\frac{1}{2}i +\frac{1}{2}|i|\right)}\\
 &\hskip 10mm  \times \Bigg\{\frac{\mathcal{E}_i}{ \Gamma\left(\frac{1}{2}b-  \frac{1}{2}a + \frac{1}{2}\right)\,\Gamma\left(\frac{1}{2}b + \frac{1}{2}a-\left[\frac{i+1}{2}\right]
  \right) } \\
  &\hskip 30mm  + \frac{\mathcal{F}_i}{ \Gamma\left(\frac{1}{2}b- \frac{1}{2}a \right)\,\Gamma\left(\frac{1}{2}b +\frac{1}{2}a-\frac{1}{2} -\left[\frac{i}{2}\right]
  \right) } \Bigg\},
 \endaligned
\end{equation}
for $i=0,\, \pm 1,\, \pm 2,\, \pm 3,\, \pm 4,\, \pm 5$.  The coefficients $ \mathcal{E}_i$ and  $ \mathcal{F}_i$  are given in the
Table \ref{table-Ei-Fi}.

\vskip 3mm

We conclude this section by remaking that if we set $i=0$ in  \eqref{G-KummerST-Ra-Ra-1} or \eqref{G-KummerST-Ra-Ra-2},
\eqref{G-Gauss-SST-Ra-Ra-1} or \eqref{G-Gauss-SST-Ra-Ra-2},
\eqref{G-Bailey-ST-Ra-Ra-1} or \eqref{G-Bailey-ST-Ra-Ra-2}, we recover the following classical Kummer, Gauss second and Baily summation theorems, respectively (see, e.g., \cite{Rain}).
\begin{equation}\label{G-KummerST} 
  {}_2F_1 \left[ \aligned a,\,b&\,;\\
                      1+a-b &\,; \endaligned
                  \,\,-1 \right]   =\frac{ \Gamma \left(1+ \frac{1}{2}a\right)\,\Gamma (1+a-b)}{\Gamma(1+a)\,
                  \Gamma \left(1+ \frac{1}{2}a-b\right)},
\end{equation}
\begin{equation}\label{G-GaussST} 
  {}_2F_1 \left[ \aligned a,\,b&\,;\\
                     \frac{1}{2}(a+b+1) &\,; \endaligned
                  \,\, \frac{1}{2} \right]   =\frac{ \Gamma \left( \frac{1}{2}\right)\, \Gamma \left( \frac{a+b+1}{2}\right)}
                  {\Gamma \left( \frac{a+1}{2}\right)\,   \Gamma \left(\frac{b+1}{2}\right)},
\end{equation}
\begin{equation}\label{G-GaussST} 
  {}_2F_1 \left[ \aligned a,\,1-a&\,;\\
                     b &\,; \endaligned
                  \,\, \frac{1}{2} \right]   =\frac{ \Gamma \left( \frac{1}{2}b\right)\, \Gamma \left( \frac{1}{2}b + \frac{1}{2} \right)}
                  {\Gamma \left( \frac{1}{2}b + \frac{1}{2}a\right)\,   \Gamma \left( \frac{1}{2}b- \frac{1}{2}a+ \frac{1}{2} \right)}.
\end{equation}

\section{General summation formulas for the Kamp\'e de F\'eriet function} \label{sec-3}

Here, thirty two summation formulas for the Kamp\'e de F\'eriet function are provided in
Theorems \ref{new-thm-1} to  \ref{new-thm-16}.

\vskip 3mm
\begin{theorem}\label{new-thm-1}
  Let $i \in \mathbb{N}_0$. Then
  \begin{equation}\label{new-thm-1-eq1}
  \aligned
  & F^{1:1;2}_{\,1:0;1} \left[\begin{array}{rrr}
                            \alpha   : &\epsilon \,;\,&\beta- \epsilon,\, 1-\alpha-\epsilon+i \,; \\
                             \beta: &\overline{\hspace{3mm}} \,;\,&1-\alpha-\epsilon+ \beta+i  \,;
                             \end{array} \frac{1}{2},\,\frac{1}{2}
   \right] \\
  & = \frac{2^{1-\alpha+i}\,\Gamma\left(\beta-\epsilon-\alpha+i+1\right)\,\Gamma\left(\alpha-i\right) }{\Gamma\left(1-2\alpha-\epsilon+\beta+i\right)\,\Gamma\left(\alpha\right) }\,
   \sum_{r=0}^{i}\,(-1)^r\,\binom{i}{r}\, \frac{\Gamma\left(\frac{1-2\alpha-\epsilon+\beta+i+r}{2}\right)}
      {\Gamma\left(\frac{\beta-\epsilon-i+r+1}{2}\right)}
  \endaligned
\end{equation}
and
    \begin{equation}\label{new-thm-1-eq2}
  \aligned
  & F^{1:1;2}_{\,1:0;1} \left[\begin{array}{rrr}
                            \alpha   : &\epsilon \,;\,&\beta- \epsilon,\, 1-\alpha-\epsilon-i \,; \\
                             \beta: &\overline{\hspace{3mm}} \,;\,&1-\alpha-\epsilon+ \beta-i  \,;
                             \end{array} \frac{1}{2},\,\frac{1}{2}
   \right] \\
  &\hskip 5mm  = \frac{2^{1-\alpha-i}\,\Gamma\left(\beta-\epsilon-\alpha-i+1\right)}
  {\Gamma\left(1-2\alpha-\epsilon+\beta-i\right)}\,
   \sum_{r=0}^{i}\,\binom{i}{r}\, \frac{\Gamma\left(\frac{1-2\alpha-\epsilon+\beta-i+r}{2}\right)}
      {\Gamma\left(\frac{\beta-\epsilon-i+r+1}{2}\right)}.
  \endaligned
\end{equation}

\end{theorem}

\begin{proof}
  Setting $x=\frac{1}{2}$ and $\gamma=1-\alpha-\epsilon+i$ $\left(i \in \mathbb{N}_0\right)$
  in \eqref{eq2-1}, we get
  \begin{equation}\label{new-thm-1-pf1}
  \aligned
    & F^{1:1;2}_{\,1:0;1} \left[\begin{array}{rrr}
                            \alpha   : &\epsilon \,;\,&\beta- \epsilon,\, 1-\alpha-\epsilon+i \,; \\
                             \beta: &\overline{\hspace{3mm}} \,;\,&1-\alpha-\epsilon+ \beta+i  \,;
                             \end{array} \frac{1}{2},\,\frac{1}{2}  \right] \\
  &\hskip 10mm = 2^{\epsilon+\alpha-\beta}\, {}_2F_1 \left[ \aligned \beta-\epsilon,\,1-2\alpha-\epsilon+\beta+i&\,;\\
                    1-\alpha-\epsilon+\beta+i &\,; \endaligned
                  \,\,\frac{1}{2} \right].
 \endaligned
  \end{equation}
 Now, the ${}_2F_1$ in the right side of \eqref{new-thm-1-pf1} can be evaluated with the help of
 the result \eqref{G-Gauss-SST-Ra-Ra-1} by taking $a=\beta-\epsilon$ and $b=1-2\alpha-\epsilon+\beta+i$.
 After some simplification, we get the result \eqref{new-thm-1-eq1}.

\vskip 3mm

The proof of the formula \eqref{new-thm-1-eq2} would run parallel to that of \eqref{new-thm-1-eq1}
by setting $x=\frac{1}{2}$ and $\gamma=1-\alpha-\epsilon-i$ $\left(i \in \mathbb{N}_0\right)$
  in \eqref{eq2-1} with the aid of the result \eqref{G-Gauss-SST-Ra-Ra-2}.
We omit the details.
\end{proof}

\vskip 3mm
\begin{theorem}\label{new-thm-2}
  Let $i \in \mathbb{N}_0$. Then
  \begin{equation}\label{new-thm-2-eq1}
  \aligned
  & F^{1:1;2}_{\,1:0;1} \left[\begin{array}{rrr}
                            \alpha   : &\epsilon \,;\,&\beta- \epsilon,\, 1+\alpha-2\beta+\epsilon+i \,; \\
                             \beta: &\overline{\hspace{3mm}} \,;\,&1+\alpha-\beta+\epsilon+ i  \,;
                             \end{array} \frac{1}{2},\,\frac{1}{2}
   \right] \\
  &\hskip 5mm  = \frac{2^{i+\alpha -2 \beta +2\epsilon}\,\Gamma\left(\beta-\epsilon-i\right)\,\Gamma\left(1+\alpha-\beta+\epsilon+i\right) }{\Gamma\left(\beta-\epsilon\right)\,\Gamma\left(1+\alpha-2\beta+2\epsilon+i\right)\, } \\
  &\hskip 10mm  \times
   \sum_{r=0}^{i}\,(-1)^r\,\binom{i}{r}\, \frac{\Gamma\left(\epsilon-\beta+\frac{1+\alpha+i+r}{2}\right)}
      {\Gamma\left(\frac{1+\alpha-i+r}{2}\right)}
  \endaligned
\end{equation}
and
    \begin{equation}\label{new-thm-2-eq2}
 \aligned
  & F^{1:1;2}_{\,1:0;1} \left[\begin{array}{rrr}
                            \alpha   : &\epsilon \,;\,&\beta- \epsilon,\, 1+\alpha-2\beta+\epsilon-i \,; \\
                             \beta: &\overline{\hspace{3mm}} \,;\,&1+\alpha-\beta+\epsilon- i  \,;
                             \end{array} \frac{1}{2},\,\frac{1}{2}
   \right] \\
  &\hskip 5mm  = \frac{2^{-i+\alpha -2 \beta +2\epsilon}\,\Gamma\left(1+\alpha-\beta+\epsilon-i\right) }{\Gamma\left(1+\alpha-2\beta+2\epsilon-i\right)\, } \\
  &\hskip 10mm  \times
   \sum_{r=0}^{i}\,\binom{i}{r}\, \frac{\Gamma\left(\epsilon-\beta+\frac{1+\alpha-i+r}{2}\right)}
      {\Gamma\left(\frac{1+\alpha-i+r}{2}\right)}.
  \endaligned
\end{equation}

\end{theorem}

\begin{proof}
 Similarly in the proof of Theorem \ref{new-thm-1}, we can establish the results here.
 Setting (i) $x=\frac{1}{2}$ and $\gamma=1+\alpha-2\beta+\epsilon +i$ $\left(i \in \mathbb{N}_0\right)$
  and (ii) $x=\frac{1}{2}$ and $\gamma=1+\alpha-2\beta+\epsilon -i$ $\left(i \in \mathbb{N}_0\right)$
  in \eqref{eq2-1} with the help of \eqref{G-Bailey-ST-Ra-Ra-1} and \eqref{G-Bailey-ST-Ra-Ra-2} yields,
  respectively,  \eqref{new-thm-2-eq1} and \eqref{new-thm-2-eq2}. We omit the details.
\end{proof}

\vskip 3mm

\begin{theorem}\label{new-thm-3}
  Let $i \in \mathbb{N}_0$. Then
  \begin{equation}\label{new-thm-3-eq1}
  \aligned
  & F^{1:1;2}_{\,1:0;1} \left[\begin{array}{rrr}
                            \alpha   : &\beta-2-i \,;\,&2+i,\,\frac{1}{2} \alpha+ 1 \,; \\
                             \beta: &\overline{\hspace{3mm}} \,;\,& \frac{1}{2} \alpha  \,;
                             \end{array} -1,\,-1
   \right] \\
  &  = \frac{2^{-2-\alpha}\,\Gamma\left(\frac{1}{2}\beta\right)}
  {(i+1)!\,\Gamma\left(\frac{1}{2}\beta-i-2\right)} \,
   \sum_{r=0}^{i}\,(-1)^r\,\binom{i}{r}\, \frac{\Gamma\left(\frac{1}{4}\beta-1+\frac{r-i}{2}\right)}
      {\Gamma\left(\frac{1}{4}\beta+1+\frac{r-i}{2}\right)}
  \endaligned
\end{equation}
and
    \begin{equation}\label{new-thm-3-eq2}
 \aligned
  & F^{1:1;2}_{\,1:0;1} \left[\begin{array}{rrr}
                            \alpha   : &\beta-2+i \,;\,&2-i,\,\frac{1}{2} \alpha+ 1 \,; \\
                             \beta: &\overline{\hspace{3mm}} \,;\,& \frac{1}{2} \alpha  \,;
                             \end{array} -1,\,-1
   \right] \\
  &\hskip 5mm   = \frac{2^{-2-\alpha}\,\Gamma\left(\frac{1}{2}\beta\right)}
  {\Gamma\left(\frac{1}{2}\beta+i-2\right)} \,
   \sum_{r=0}^{i}\,\binom{i}{r}\, \frac{\Gamma\left(\frac{1}{4}\beta-1+\frac{r+i}{2}\right)}
      {\Gamma\left(\frac{1}{4}\beta+1+\frac{r-i}{2}\right)}.
  \endaligned
\end{equation}

\end{theorem}

\begin{proof}
 Similarly in the proof of Theorem \ref{new-thm-1}, we can establish the results here.
 Setting (i) $x=-1$ and $\epsilon=\beta -2 -i$ $\left(i \in \mathbb{N}_0\right)$
  and (ii) $x=-1$ and $\epsilon=\beta -2 +i$ $\left(i \in \mathbb{N}_0\right)$
  in \eqref{eq2-2} with the help of \eqref{G-KummerST-Ra-Ra-1} and \eqref{G-KummerST-Ra-Ra-2} yields,
  respectively,  \eqref{new-thm-3-eq1} and \eqref{new-thm-3-eq2}. We omit the details.
\end{proof}

\vskip 3mm

\begin{theorem}\label{new-thm-4}
  Let $i \in \mathbb{N}_0$. Then
  \begin{equation}\label{new-thm-4-eq1}
  \aligned
  & F^{1:1;2}_{\,1:0;1} \left[\begin{array}{rrr}
                           \alpha : & \frac{1}{2} \beta +2+i   \,;\,&\frac{1}{2} \beta -2-i,\,\frac{1}{2}\alpha+1 \,; \\
                             \beta: &\overline{\hspace{3mm}} \,;\,& \frac{1}{2} \alpha  \,;
                             \end{array} \frac{1}{2},\,\frac{1}{2}
   \right] \\
  &  = \frac{(-1)^i\, 2^{\alpha+3+i}}
  {\beta\,(i+1)!} \,
   \sum_{r=0}^{i}\,(-1)^r\,\binom{i}{r}\, \frac{\Gamma\left(\frac{1}{4}\beta+\frac{r+1}{2}\right)}
      {\Gamma\left(\frac{1}{4}\beta -i +\frac{r-1}{2}\right)}
  \endaligned
\end{equation}
and
    \begin{equation}\label{new-thm-4-eq2}
  \aligned
  & F^{1:1;2}_{\,1:0;1} \left[\begin{array}{rrr}
                           \alpha : & \frac{1}{2} \beta +2-i   \,;\,&\frac{1}{2} \beta -2+i,\,\frac{1}{2}\alpha+1 \,; \\                             \beta: &\overline{\hspace{3mm}} \,;\,& \frac{1}{2} \alpha  \,;
                             \end{array} \frac{1}{2},\,\frac{1}{2}
   \right] \\
  & \hskip 5mm  = \frac{2^{\alpha+3-i}}
  {\beta} \,
   \sum_{r=0}^{i}\,\binom{i}{r}\, \frac{\Gamma\left(\frac{1}{4}\beta+\frac{r+1}{2}\right)}
      {\Gamma\left(\frac{1}{4}\beta  +\frac{r-1}{2}\right)}.
  \endaligned
\end{equation}

\end{theorem}

\begin{proof}
 Similarly in the proof of Theorem \ref{new-thm-1}, we can establish the results here.
 Setting (i) $x=\frac{1}{2}$ and $\epsilon=\frac{1}{2}\beta +2 +i$ $\left(i \in \mathbb{N}_0\right)$
  and (ii) $x=\frac{1}{2}$ and $\epsilon=\frac{1}{2}\beta +2 -i$ $\left(i \in \mathbb{N}_0\right)$
  in \eqref{eq2-2} with the help of \eqref{G-KummerST-Ra-Ra-1} and \eqref{G-KummerST-Ra-Ra-2} yields,
  respectively,  \eqref{new-thm-4-eq1} and \eqref{new-thm-4-eq2}. We omit the details.
\end{proof}

\vskip 3mm

\begin{theorem}\label{new-thm-5}
  Let $i \in \mathbb{N}_0$. Then
  \begin{equation}\label{new-thm-5-eq1}
  \aligned
  & F^{1:1;3}_{\,1:0;2} \left[\begin{array}{rrr}
                           \alpha : & \alpha+\beta-i   \,;\,&i-\alpha,\,1+\frac{1}{2}\alpha,\, \frac{\alpha-\beta}{2} \,; \\
                             \beta: &\overline{\hspace{3mm}} \,;\,& \frac{1}{2} \alpha,\, 1+ \frac{\alpha+\beta}{2} \,;
                             \end{array} -1,\,-1
   \right] \\
  &  = \frac{\Gamma\left(-\alpha\right)\,\Gamma\left(1+ \frac{\alpha+\beta}{2}\right) }
  {\Gamma\left(i-\alpha\right)\,\Gamma\left(\frac{1}{2}\beta+\frac{3}{2}\alpha-i+1 \right)} \,
   \sum_{r=0}^{i}\,(-1)^r\,\binom{i}{r}\, \frac{\Gamma\left(\frac{1}{4}\beta +\frac{3}{4}\alpha +\frac{r+1-i}{2}\right)}
      {\Gamma\left(\frac{1}{4}\beta -\frac{1}{4}\alpha +\frac{r+1-i}{2}\right)}
  \endaligned
\end{equation}
and
    \begin{equation}\label{new-thm-5-eq2}
 \aligned
  & F^{1:1;3}_{\,1:0;2} \left[\begin{array}{rrr}
                           \alpha : & \alpha+\beta+i   \,;\,&-\alpha-i,\,1+\frac{1}{2}\alpha,\, \frac{\alpha-\beta}{2} \,; \\
                             \beta: &\overline{\hspace{3mm}} \,;\,& \frac{1}{2} \alpha,\, 1+ \frac{\alpha+\beta}{2} \,;
                             \end{array} -1,\,-1
   \right] \\
  & \hskip 5mm = \frac{\Gamma\left(1+ \frac{\alpha+\beta}{2}\right) }
  {\Gamma\left(\frac{1}{2}\beta+\frac{3}{2}\alpha+i+1 \right)} \,
   \sum_{r=0}^{i}\,\binom{i}{r}\, \frac{\Gamma\left(\frac{1}{4}\beta +\frac{3}{4}\alpha +\frac{r+1+i}{2}\right)}
      {\Gamma\left(\frac{1}{4}\beta -\frac{1}{4}\alpha +\frac{r+1-i}{2}\right)}.
  \endaligned
\end{equation}

\end{theorem}

\begin{proof}
 Similarly in the proof of Theorem \ref{new-thm-1}, we can establish the results here.
 Setting (i) $x=-1$ and $\epsilon=\alpha+\beta -i$ $\left(i \in \mathbb{N}_0\right)$
  and (ii) $x=-1$ and $\epsilon=\alpha+\beta +i$ $\left(i \in \mathbb{N}_0\right)$
  in \eqref{eq2-3} with the help of \eqref{G-KummerST-Ra-Ra-1} and \eqref{G-KummerST-Ra-Ra-2} yields,
  respectively,  \eqref{new-thm-5-eq1} and \eqref{new-thm-5-eq2}. We omit the details.
\end{proof}

\vskip 3mm

\begin{theorem}\label{new-thm-6}
  Let $i \in \mathbb{N}_0$. Then
  \begin{equation}\label{new-thm-6-eq1}
  \aligned
  & F^{1:1;3}_{\,1:0;2} \left[\begin{array}{rrr}
                           \alpha : & \frac{1}{2}\beta-\frac{3}{2}\alpha-1+i   \,;\,&\frac{1}{2}\beta+\frac{3}{2}\alpha+1-i ,\, 1+\frac{1}{2}\alpha,\, \frac{\alpha-\beta}{2} \,; \\
                             \beta: &\overline{\hspace{3mm}} \,;\,& \frac{1}{2} \alpha,\, 1+ \frac{\alpha+\beta}{2} \,;
                             \end{array} \frac{1}{2},\,\frac{1}{2}
   \right] \\
  &  = \frac{2^{i-\alpha-2}\,\Gamma\left(\alpha-i+1\right)\,\Gamma\left(1+ \frac{\alpha+\beta}{2}\right) }
  {\Gamma\left(\alpha+1\right)\,\Gamma\left(\frac{\beta-\alpha}{2} \right)} \,
   \sum_{r=0}^{i}\,(-1)^r\,\binom{i}{r}\, \frac{\Gamma\left(\frac{\beta-\alpha}{4} +\frac{1}{2}r\right)}
      {\Gamma\left(\frac{\beta+3\alpha}{4}+1-i+\frac{1}{2}r \right)}
  \endaligned
\end{equation}
and
    \begin{equation}\label{new-thm-6-eq2}
 \aligned
  & F^{1:1;3}_{\,1:0;2} \left[\begin{array}{rrr}
                           \alpha : & \frac{1}{2}\beta-\frac{3}{2}\alpha-i-1   \,;\,&\frac{1}{2}\beta+\frac{3}{2}\alpha+1+i ,\, 1+\frac{1}{2}\alpha,\, \frac{\alpha-\beta}{2} \,; \\
                             \beta: &\overline{\hspace{3mm}} \,;\,& \frac{1}{2} \alpha,\, 1+ \frac{\alpha+\beta}{2} \,;
                             \end{array} \frac{1}{2},\,\frac{1}{2}
   \right] \\
  &  = \frac{2^{-i-\alpha-2}\,\Gamma\left(1+ \frac{\alpha+\beta}{2}\right) }
  {\Gamma\left(\frac{\beta-\alpha}{2} \right)} \,
   \sum_{r=0}^{i}\,\binom{i}{r}\, \frac{\Gamma\left(\frac{\beta-\alpha}{4} +\frac{1}{2}r\right)}
      {\Gamma\left(\frac{\beta+3\alpha}{4}+1+\frac{1}{2}r \right)}.
  \endaligned
\end{equation}

\end{theorem}

\begin{proof}
 Similarly in the proof of Theorem \ref{new-thm-1}, we can establish the results here.
 Setting (i) $x=\frac{1}{2}$ and $\epsilon=\frac{1}{2}\beta-\frac{3}{2}\alpha-1+i$ $\left(i \in \mathbb{N}_0\right)$
  and (ii) $x=\frac{1}{2}$ and $\epsilon=\frac{1}{2}\beta-\frac{3}{2}\alpha-1-i$ $\left(i \in \mathbb{N}_0\right)$
  in \eqref{eq2-3} with the help of \eqref{G-Gauss-SST-Ra-Ra-1} and \eqref{G-Gauss-SST-Ra-Ra-2} yields,
  respectively,  \eqref{new-thm-6-eq1} and \eqref{new-thm-6-eq2}. We omit the details.
\end{proof}

\vskip 3mm

\begin{theorem}\label{new-thm-7}
  Let $i \in \mathbb{N}_0$. Then
  \begin{equation}\label{new-thm-7-eq1}
  \aligned
  & F^{1:1;3}_{\,1:0;2} \left[\begin{array}{rrr}
                           \alpha : & \frac{3}{2}\beta-\frac{1}{2}\alpha-1-i   \,;\,&\frac{\alpha-\beta}{2}+1+i ,\, 1+\frac{1}{2}\alpha,\, \frac{\alpha-\beta}{2} \,; \\
                             \beta: &\overline{\hspace{3mm}} \,;\,& \frac{1}{2} \alpha,\, 1+ \frac{\alpha+\beta}{2} \,;
                             \end{array} \frac{1}{2},\,\frac{1}{2}
   \right] \\
  &  = \frac{2^{\alpha-1}\,\Gamma\left(\frac{1}{2}\beta-\frac{1}{2}\alpha-i\right)\,\Gamma\left(1+ \frac{\alpha+\beta}{2}\right) }
  {\Gamma\left(\frac{1}{2}\beta-\frac{1}{2}\alpha \right)\,\Gamma\left(1+\alpha \right)} \,
   \sum_{r=0}^{i}\,(-1)^r\,\binom{i}{r}\, \frac{\Gamma\left(\frac{\alpha+r+1}{2}\right)}
      {\Gamma\left(\frac{\beta+r+1}{2}-i\right)}
  \endaligned
\end{equation}
and
    \begin{equation}\label{new-thm-7-eq2}
 \aligned
  & F^{1:1;3}_{\,1:0;2} \left[\begin{array}{rrr}
                           \alpha : & \frac{3}{2}\beta-\frac{1}{2}\alpha-1+i   \,;\,&\frac{\alpha-\beta}{2}+1-i ,\, 1+\frac{1}{2}\alpha,\, \frac{\alpha-\beta}{2} \,; \\
                             \beta: &\overline{\hspace{3mm}} \,;\,& \frac{1}{2} \alpha,\, 1+ \frac{\alpha+\beta}{2} \,;
                             \end{array} \frac{1}{2},\,\frac{1}{2}
   \right] \\
  & \hskip 5mm  = \frac{2^{\alpha-1}\,\Gamma\left(1+ \frac{\alpha+\beta}{2}\right) }
  {\Gamma\left(1+\alpha \right)} \,
   \sum_{r=0}^{i}\,\binom{i}{r}\, \frac{\Gamma\left(\frac{\alpha+r+1}{2}\right)}
      {\Gamma\left(\frac{\beta+r+1}{2}\right)}.
  \endaligned
\end{equation}

\end{theorem}

\begin{proof}
 Similarly in the proof of Theorem \ref{new-thm-1}, we can establish the results here.
 Setting (i) $x=\frac{1}{2}$ and $\epsilon=\frac{3}{2}\beta-\frac{1}{2}\alpha-1-i$ $\left(i \in \mathbb{N}_0\right)$
  and (ii) $x=\frac{1}{2}$ and $\epsilon=\frac{3}{2}\beta-\frac{1}{2}\alpha-1+i$ $\left(i \in \mathbb{N}_0\right)$
  in \eqref{eq2-3} with the help of \eqref{G-Bailey-ST-Ra-Ra-1} and \eqref{G-Bailey-ST-Ra-Ra-2} yields,
  respectively,  \eqref{new-thm-7-eq1} and \eqref{new-thm-7-eq2}. We omit the details.
\end{proof}

\vskip 3mm

\begin{theorem}\label{new-thm-8}
  Let $i \in \mathbb{N}_0$. Then
  \begin{equation}\label{new-thm-8-eq1}
  \aligned
  &  F^{0:2;2}_{\,1:0;0} \left[\begin{array}{rrr}
                               \overline{\hspace{3mm}}: &\alpha,\, \epsilon \,;\,& \beta-\epsilon,\,1-\alpha-\epsilon+i \,; \\
                             \beta: &\overline{\hspace{3mm}} \,;\,& \overline{\hspace{3mm}}  \,;
                             \end{array} \frac{1}{2},\,-1
   \right] \\
   & \hskip 5mm = F_3 \left(\alpha,\, \beta-\epsilon\,:\,\epsilon,\,1-\alpha-\epsilon+i\,;\,\beta\,;\,  \frac{1}{2},\, -1 \right)\\
  &  = \frac{2^{\alpha-i+2\epsilon-2}\,\Gamma\left(1-\epsilon\right)\,\Gamma\left(\beta\right) }
  {\Gamma\left(1-\epsilon+i \right)\,\Gamma\left(\beta+\epsilon-i-1 \right)} \,
   \sum_{r=0}^{i}\,(-1)^r\,\binom{i}{r}\, \frac{\Gamma\left(\frac{\beta+\epsilon-i+r-1}{2}\right)}
      {\Gamma\left(\frac{\beta-\epsilon-i+r+1}{2}\right)}
  \endaligned
\end{equation}
and
    \begin{equation}\label{new-thm-8-eq2}
 \aligned
  & F^{0:2;2}_{\,1:0;0} \left[\begin{array}{rrr}
                               \overline{\hspace{3mm}}: &\alpha,\, \epsilon \,;\,& \beta-\epsilon,\,1-\alpha-\epsilon-i \,; \\
                             \beta: &\overline{\hspace{3mm}} \,;\,& \overline{\hspace{3mm}}  \,;
                             \end{array} \frac{1}{2},\,-1
   \right] \\
   & \hskip 5mm = F_3 \left(\alpha,\, \beta-\epsilon\,:\,\epsilon,\,1-\alpha-\epsilon-i\,;\,\beta\,;\,  \frac{1}{2},\, -1 \right)\\
  & \hskip 5mm  = \frac{2^{\alpha+i+2\epsilon-2}\,\Gamma\left(\beta\right) }
  {\Gamma\left(\beta+\epsilon+i-1 \right)} \,
   \sum_{r=0}^{i}\,\binom{i}{r}\, \frac{\Gamma\left(\frac{\beta+\epsilon+i+r-1}{2}\right)}
      {\Gamma\left(\frac{\beta-\epsilon-i+r+1}{2}\right)}.
  \endaligned
\end{equation}

\end{theorem}

\begin{proof}
 Similarly in the proof of Theorem \ref{new-thm-1}, we can establish the results here.
 Setting (i) $x=\frac{1}{2}$ and $\gamma=1-\alpha-\epsilon+i$ $\left(i \in \mathbb{N}_0\right)$
  and (ii) $x=\frac{1}{2}$ and $\gamma=1-\alpha-\epsilon-i$ $\left(i \in \mathbb{N}_0\right)$
  in \eqref{eq2-4} with the help of \eqref{G-KummerST-Ra-Ra-1} and \eqref{G-KummerST-Ra-Ra-2} yields,
  respectively,  \eqref{new-thm-8-eq1} and \eqref{new-thm-8-eq2}. We omit the details.
\end{proof}

\vskip 3mm

\begin{theorem}\label{new-thm-9}
  Let $i \in \mathbb{N}_0$. Then
  \begin{equation}\label{new-thm-9-eq1}
  \aligned
 &  F^{0:2;2}_{\,1:0;0} \left[\begin{array}{rrr}
                               \overline{\hspace{3mm}}: &\alpha,\, \epsilon \,;\,& \beta-\epsilon,\,\beta+\epsilon-\alpha-i-1 \,; \\
                             \beta: &\overline{\hspace{3mm}} \,;\,& \overline{\hspace{3mm}}  \,;
                             \end{array} -1,\,\frac{1}{2}
   \right] \\
   & \hskip 5mm = F_3 \left(\alpha,\, \beta-\epsilon\,:\,\epsilon,\,\beta+\epsilon-\alpha-i-1\,;\,\beta\,;\, -1,\,\frac{1}{2} \right)\\
  &  = \frac{2^{\beta+\epsilon-\alpha-i-2}\,\Gamma\left(1-\epsilon\right)\,\Gamma\left(\beta\right) }
  {\Gamma\left(1-\epsilon+i \right)\,\Gamma\left(\beta+\epsilon-i-1 \right)} \,
   \sum_{r=0}^{i}\,(-1)^r\,\binom{i}{r}\, \frac{\Gamma\left(\frac{\beta+\epsilon-i+r-1}{2}\right)}
      {\Gamma\left(\frac{\beta-\epsilon-i+r+1}{2}\right)}
  \endaligned
\end{equation}
and
    \begin{equation}\label{new-thm-9-eq2}
 \aligned
 &  F^{0:2;2}_{\,1:0;0} \left[\begin{array}{rrr}
                               \overline{\hspace{3mm}}: &\alpha,\, \epsilon \,;\,& \beta-\epsilon,\,\beta+\epsilon-\alpha+i-1 \,; \\
                             \beta: &\overline{\hspace{3mm}} \,;\,& \overline{\hspace{3mm}}  \,;
                             \end{array} -1,\,\frac{1}{2}
   \right] \\
   & \hskip 5mm = F_3 \left(\alpha,\, \beta-\epsilon\,:\,\epsilon,\,\beta+\epsilon-\alpha+i-1\,;\,\beta\,;\, -1,\,\frac{1}{2} \right)\\
  & \hskip 5mm  = \frac{2^{\beta+\epsilon-\alpha+i-2}\,\Gamma\left(\beta\right) }
  {\Gamma\left(\beta+\epsilon+i-1 \right)} \,
   \sum_{r=0}^{i}\,\binom{i}{r}\, \frac{\Gamma\left(\frac{\beta+\epsilon+i+r-1}{2}\right)}
      {\Gamma\left(\frac{\beta-\epsilon-i+r+1}{2}\right)}.
  \endaligned
\end{equation}

\end{theorem}

\begin{proof}
 Similarly in the proof of Theorem \ref{new-thm-1}, we can establish the results here.
 Setting (i) $x=-1$ and $\gamma=\beta+\epsilon-\alpha-i-1$ $\left(i \in \mathbb{N}_0\right)$
  and (ii) $x=-1$ and $\gamma=\beta+\epsilon-\alpha+i-1$ $\left(i \in \mathbb{N}_0\right)$
  in \eqref{eq2-4} with the help of \eqref{G-KummerST-Ra-Ra-1} and \eqref{G-KummerST-Ra-Ra-2} yields,
  respectively,  \eqref{new-thm-9-eq1} and \eqref{new-thm-9-eq2}. We omit the details.
\end{proof}

\vskip 3mm

\begin{theorem}\label{new-thm-10}
  Let $i \in \mathbb{N}_0$. Then
  \begin{equation}\label{new-thm-10-eq1}
  \aligned
  &  F^{0:2;2}_{\,1:0;0} \left[\begin{array}{rrr}
                               \overline{\hspace{3mm}}: &\alpha,\, \epsilon \,;\,& \beta-\epsilon,\,1-\alpha-\beta+\epsilon+i \,; \\
                             \beta: &\overline{\hspace{3mm}} \,;\,& \overline{\hspace{3mm}}  \,;
                             \end{array} -1,\,\frac{1}{2}
   \right] \\
   & \hskip 5mm = F_3 \left(\alpha,\, \beta-\epsilon\,:\,\epsilon,\,1-\alpha-\beta+\epsilon+i\,;\,\beta\,;\,  -1,\,\frac{1}{2} \right)\\
  &  = \frac{2^{\epsilon-\alpha-\beta +i}\,\Gamma\left(\beta-\epsilon+i\right)\,\Gamma\left(\beta\right) }
  {\Gamma\left(\epsilon \right)\,\Gamma\left(\beta-\epsilon \right)} \,
   \sum_{r=0}^{i}\,(-1)^r\,\binom{i}{r}\, \frac{\Gamma\left(\frac{\epsilon+r}{2}\right)}
      {\Gamma\left(\beta-i+\frac{r-\epsilon}{2}\right)}
  \endaligned
\end{equation}
and
    \begin{equation}\label{new-thm-10-eq2}
 \aligned
   &  F^{0:2;2}_{\,1:0;0} \left[\begin{array}{rrr}
                               \overline{\hspace{3mm}}: &\alpha,\, \epsilon \,;\,& \beta-\epsilon,\,1-\alpha-\beta+\epsilon-i \,; \\
                             \beta: &\overline{\hspace{3mm}} \,;\,& \overline{\hspace{3mm}}  \,;
                             \end{array} -1,\,\frac{1}{2}
   \right] \\
   & \hskip 5mm = F_3 \left(\alpha,\, \beta-\epsilon\,:\,\epsilon,\,1-\alpha-\beta+\epsilon-i\,;\,\beta\,;\,  -1,\,\frac{1}{2} \right)\\
  & \hskip 5mm  = \frac{2^{\epsilon-\alpha-\beta -i}\,\Gamma\left(\beta \right)}
  {\Gamma\left(\epsilon \right)} \,
   \sum_{r=0}^{i}\,\binom{i}{r}\, \frac{\Gamma\left(\frac{\epsilon+r}{2}\right)}
      {\Gamma\left(\beta+\frac{r-\epsilon}{2}\right)}.
  \endaligned
\end{equation}

\end{theorem}

\begin{proof}
 Similarly in the proof of Theorem \ref{new-thm-1}, we can establish the results here.
 Setting (i) $x=-1$ and $\gamma=1-\alpha-\beta+\epsilon+i$ $\left(i \in \mathbb{N}_0\right)$
  and (ii) $x=-1$ and $\gamma=1-\alpha-\beta+\epsilon-i$ $\left(i \in \mathbb{N}_0\right)$
  in \eqref{eq2-4} with the help of \eqref{G-Bailey-ST-Ra-Ra-1} and \eqref{G-Bailey-ST-Ra-Ra-2} yields,
  respectively,  \eqref{new-thm-10-eq1} and \eqref{new-thm-10-eq2}. We omit the details.
\end{proof}

\vskip 3mm

\begin{theorem}\label{new-thm-11}
  Let $i \in \mathbb{N}_0$. Then
  \begin{equation}\label{new-thm-11-eq1}
  \aligned
  &  F^{2:0;1}_{\,1:0;1} \left[\begin{array}{rrr}
                            \alpha,\, \gamma   : &\overline{\hspace{3mm}} \,;\,& \alpha-\beta-\gamma+1+i \,; \\
                             \beta: &\overline{\hspace{3mm}} \,;\,& \alpha-\gamma+1+i  \,;
                             \end{array} -1,\,1
   \right] \\
  &  = \frac{2^{i-2\gamma}\,\Gamma\left(1+\alpha-\gamma+i\right)\,\Gamma\left(\gamma-i\right) }
  {\Gamma\left(\gamma \right)\,\Gamma\left(1+\alpha-2\gamma+i \right)} \,
   \sum_{r=0}^{i}\,(-1)^r\,\binom{i}{r}\, \frac{\Gamma\left(\frac{\alpha+i+r+1}{2}-\gamma\right)}
      {\Gamma\left(\frac{\alpha-i+r+1}{2}\right)}
  \endaligned
\end{equation}
and
    \begin{equation}\label{new-thm-11-eq2}
 \aligned
   &  F^{2:0;1}_{\,1:0;1} \left[\begin{array}{rrr}
                            \alpha,\, \gamma   : &\overline{\hspace{3mm}} \,;\,& \alpha-\beta-\gamma+1-i \,; \\
                             \beta: &\overline{\hspace{3mm}} \,;\,& \alpha-\gamma+1-i  \,;
                             \end{array} -1,\,1
   \right] \\
  & \hskip 5mm  =  \frac{2^{-i-2\gamma}\,\Gamma\left(1+\alpha-\gamma-i\right) }
  {\Gamma\left(1+\alpha-2\gamma-i \right)} \,
   \sum_{r=0}^{i}\,\binom{i}{r}\, \frac{\Gamma\left(\frac{\alpha-i+r+1}{2}-\gamma\right)}
      {\Gamma\left(\frac{\alpha-i+r+1}{2}\right)}.
  \endaligned
\end{equation}

\end{theorem}

\begin{proof}
 Similarly in the proof of Theorem \ref{new-thm-1}, we can establish the results here.
 Setting (i) $x=-1$ and $\epsilon=\alpha-\beta-r+1+i$ $\left(i \in \mathbb{N}_0\right)$
  and (ii) $x=-1$ and $\epsilon=\alpha-\beta-r+1-i$ $\left(i \in \mathbb{N}_0\right)$
  in \eqref{eq2-5} with the help of \eqref{G-Bailey-ST-Ra-Ra-1} and \eqref{G-Bailey-ST-Ra-Ra-2} yields,
  respectively,  \eqref{new-thm-11-eq1} and \eqref{new-thm-11-eq2}. We omit the details.
\end{proof}

\vskip 3mm

\begin{theorem}\label{new-thm-12}
  Let $i \in \mathbb{N}_0$. Then
  \begin{equation}\label{new-thm-12-eq1}
  \aligned
  &    F^{2:0;1}_{\,1:0;1} \left[\begin{array}{rrr}
                            \alpha,\, \gamma   : &\overline{\hspace{3mm}} \,;\,& 1-\alpha-\beta+\gamma+i \,; \\
                             \beta: &\overline{\hspace{3mm}} \,;\,& 1-\alpha+\gamma+i  \,;
                             \end{array} -1,\,1
   \right] \\
  &  = \frac{2^{i-2\alpha}\,\Gamma\left(\alpha-i\right)\,\Gamma\left(1-\alpha+i\right)}
  {\Gamma\left(\alpha \right)\,\Gamma\left(1+\gamma-2\alpha+i \right)} \,
   \sum_{r=0}^{i}\,(-1)^r\,\binom{i}{r}\, \frac{\Gamma\left(\frac{1+\gamma+i+r}{2}-\alpha\right)}
      {\Gamma\left(\frac{1+\gamma-i+r}{2}\right)}
  \endaligned
\end{equation}
and
    \begin{equation}\label{new-thm-12-eq2}
 \aligned
   &    F^{2:0;1}_{\,1:0;1} \left[\begin{array}{rrr}
                            \alpha,\, \gamma   : &\overline{\hspace{3mm}} \,;\,& 1-\alpha-\beta+\gamma-i \,; \\
                             \beta: &\overline{\hspace{3mm}} \,;\,& 1-\alpha+\gamma-i  \,;
                             \end{array} -1,\,1
   \right] \\
  & \hskip 5mm  =  \frac{2^{-i-2\alpha}\,\Gamma\left(1-\alpha+\gamma-i\right)}
  {\Gamma\left(1+\gamma-2\alpha-i \right)} \,
   \sum_{r=0}^{i}\binom{i}{r}\, \frac{\Gamma\left(\frac{1+\gamma-i+r}{2}-\alpha\right)}
      {\Gamma\left(\frac{1+\gamma-i+r}{2}\right)}.
  \endaligned
\end{equation}

\end{theorem}

\begin{proof}
 Similarly in the proof of Theorem \ref{new-thm-1}, we can establish the results here.
 Setting (i) $x=-1$ and $\epsilon=1-\alpha-\beta+\gamma+i$ $\left(i \in \mathbb{N}_0\right)$
  and (ii) $x=-1$ and $\epsilon=1-\alpha-\beta+\gamma-i$ $\left(i \in \mathbb{N}_0\right)$
  in \eqref{eq2-5} with the help of \eqref{G-Bailey-ST-Ra-Ra-1} and \eqref{G-Bailey-ST-Ra-Ra-2} yields,
  respectively,  \eqref{new-thm-12-eq1} and \eqref{new-thm-12-eq2}. We omit the details.
\end{proof}

\vskip 3mm

\begin{theorem}\label{new-thm-13}
  Let $i \in \mathbb{N}_0$. Then
  \begin{equation}\label{new-thm-13-eq1}
  \aligned
  &   F^{2:0;1}_{\,1:0;1} \left[\begin{array}{rrr}
                            \alpha,\, \gamma   : &\overline{\hspace{3mm}} \,;\,& \frac{1}{2}\gamma +1\,; \\
                            2\alpha +4+2i: &\overline{\hspace{3mm}} \,;\,& \frac{1}{2}\gamma  \,;
                             \end{array} -1,\,1
   \right] \\
  &  = \frac{(-1)^i\,2^{i+2}}
  {(\alpha+2+i)\,(i+1)!} \,
   \sum_{r=0}^{i}\,(-1)^r\,\binom{i}{r}\, \frac{\Gamma\left(\frac{\alpha+i+r+3}{2}\right)}
      {\Gamma\left(\frac{\alpha-i+r+1}{2}\right)}
  \endaligned
\end{equation}
and
    \begin{equation}\label{new-thm-13-eq2}
 \aligned
   &    F^{2:0;1}_{\,1:0;1} \left[\begin{array}{rrr}
                            \alpha,\, \gamma   : &\overline{\hspace{3mm}} \,;\,& \frac{1}{2}\gamma +1\,; \\
                            2\alpha +4-2i: &\overline{\hspace{3mm}} \,;\,& \frac{1}{2}\gamma  \,;
                             \end{array} -1,\,1
   \right] \\
  & \hskip 5mm  =  \frac{2^{-i+2}}{\alpha+2-i} \,
   \sum_{r=0}^{i}\,\binom{i}{r}\, \frac{\Gamma\left(\frac{\alpha-i+r+3}{2}\right)}
      {\Gamma\left(\frac{\alpha-i+r+1}{2}\right)}.
  \endaligned
\end{equation}

\end{theorem}

\begin{proof}
 Similarly in the proof of Theorem \ref{new-thm-1}, we can establish the results here.
 Setting (i) $x=-1$ and $\beta=2\alpha+4+2i $ $\left(i \in \mathbb{N}_0\right)$
  and (ii) $x=-1$ and $\beta=2\alpha+4-2i $  $\left(i \in \mathbb{N}_0\right)$
  in \eqref{eq2-6} with the help of \eqref{G-KummerST-Ra-Ra-1} and \eqref{G-KummerST-Ra-Ra-2} yields,
  respectively,  \eqref{new-thm-13-eq1} and \eqref{new-thm-13-eq2}. We omit the details.
\end{proof}

\vskip 3mm

\begin{theorem}\label{new-thm-14}
  Let $i \in \mathbb{N}_0$. Then
  \begin{equation}\label{new-thm-14-eq1}
  \aligned
  &   F^{2:0;2}_{\,1:0;2} \left[\begin{array}{rrr}
                            \alpha,\, i-\alpha   : &\overline{\hspace{3mm}} \,;\,& 1-\frac{1}{2}\alpha + \frac{1}{2}i,\, \frac{i-\alpha-\beta}{2} \,; \\
                             \beta: &\overline{\hspace{3mm}} \,;\,& -\frac{1}{2}\alpha + \frac{1}{2}i,\, 1+\frac{i-\alpha+\beta}{2} \,;
                             \end{array} \frac{1}{2},\,-\frac{1}{2}
   \right] \\
  &  = \frac{2^{i-\alpha}\,\Gamma\left(\alpha-i\right)\,\Gamma\left(1 +\frac{\beta-\alpha+i}{2}\right)}
     {\Gamma\left(\alpha\right)\, \Gamma\left(1+\frac{\beta-3\alpha+i}{2} \right) }\,
   \sum_{r=0}^{i}\,(-1)^r\,\binom{i}{r}\, \frac{\Gamma\left(\frac{\beta-3\alpha+i}{4}+\frac{r+1}{2}\right)}
      {\Gamma\left(\frac{\beta + \alpha -3i}{4}+\frac{r+1}{2}\right)}
  \endaligned
\end{equation}
and
    \begin{equation}\label{new-thm-14-eq2}
 \aligned
   &   F^{2:0;2}_{\,1:0;2} \left[\begin{array}{rrr}
                            \alpha,\, -\alpha-i   : &\overline{\hspace{3mm}} \,;\,& 1-\frac{1}{2}\alpha - \frac{1}{2}i,\, \frac{-i-\alpha-\beta}{2} \,; \\
                             \beta: &\overline{\hspace{3mm}} \,;\,& -\frac{1}{2}\alpha - \frac{1}{2}i,\, 1+\frac{\beta-\alpha-i}{2} \,;
                             \end{array} \frac{1}{2},\,-\frac{1}{2}
   \right] \\
  &   = \frac{2^{-i-\alpha}\,\Gamma\left(1 +\frac{\beta-\alpha-i}{2}\right)}
     {\Gamma\left(1+\frac{\beta-3\alpha-i}{2} \right) }\,
   \sum_{r=0}^{i}\,\binom{i}{r}\, \frac{\Gamma\left(\frac{\beta-3\alpha-i}{4}+\frac{r+1}{2}\right)}
      {\Gamma\left(\frac{\beta + \alpha -i}{4}+\frac{r+1}{2}\right)}.
  \endaligned
\end{equation}

\end{theorem}

\begin{proof}
 Similarly in the proof of Theorem \ref{new-thm-1}, we can establish the results here.
 Setting (i) $x=\frac{1}{2}$ and $\gamma=-\alpha +i$ $\left(i \in \mathbb{N}_0\right)$
  and (ii) $x=\frac{1}{2}$ and $\gamma=-\alpha-i$   $\left(i \in \mathbb{N}_0\right)$
  in \eqref{eq2-7} with the help of \eqref{G-KummerST-Ra-Ra-1} and \eqref{G-KummerST-Ra-Ra-2} yields,
  respectively,  \eqref{new-thm-14-eq1} and \eqref{new-thm-14-eq2}. We omit the details.
\end{proof}

\vskip 3mm

\begin{theorem}\label{new-thm-15}
  Let $i \in \mathbb{N}_0$. Then
  \begin{equation}\label{new-thm-15-eq1}
  \aligned
  &    F^{2:0;2}_{\,1:0;2} \left[\begin{array}{rrr}
                            \frac{1}{2}\beta+ \frac{3}{2}\gamma+1-i,\,\gamma   : &\overline{\hspace{3mm}} \,;\,& 1+\frac{1}{2}\gamma,\, \frac{\gamma-\beta}{2} \,; \\
                             \beta: &\overline{\hspace{3mm}} \,;\,& \frac{1}{2}\gamma ,\, 1+\frac{\gamma+\beta}{2} \,;
                             \end{array}-1,\,1
   \right] \\
  &  = \frac{2^{\frac{\beta-\gamma}{2} -\alpha-1}\,\Gamma\left(1+ \frac{\beta+\gamma}{2} \right)\,
  \Gamma\left(\gamma-i+1\right)}{\Gamma\left(\frac{\beta-\gamma}{2} \right)\,
     \Gamma\left(\gamma+1\right) }\,
   \sum_{r=0}^{i}\,(-1)^r\,\binom{i}{r}\, \frac{\Gamma\left(\frac{\beta-\gamma}{4}+\frac{r}{2}\right)}
      {\Gamma\left(\frac{\beta +3\gamma}{4}+1-i+\frac{r}{2}\right)}
  \endaligned
\end{equation}
and
    \begin{equation}\label{new-thm-15-eq2}
 \aligned
   &   F^{2:0;2}_{\,1:0;2} \left[\begin{array}{rrr}
                            \frac{1}{2}\beta+ \frac{3}{2}\gamma+1+i,\,\gamma   : &\overline{\hspace{3mm}} \,;\,& 1+\frac{1}{2}\gamma,\, \frac{\gamma-\beta}{2} \,; \\
                             \beta: &\overline{\hspace{3mm}} \,;\,& \frac{1}{2}\gamma ,\, 1+\frac{\gamma+\beta}{2} \,;
                             \end{array}-1,\,1
   \right] \\
  &   = \frac{2^{\frac{\beta-\gamma}{2} -\alpha-1}\,\Gamma\left(1+ \frac{\beta+\gamma}{2} \right)
  }{\Gamma\left(\frac{\beta-\gamma}{2} \right) }\,
   \sum_{r=0}^{i}\,\,\binom{i}{r}\, \frac{\Gamma\left(\frac{\beta-\gamma}{4}+\frac{r}{2}\right)}
      {\Gamma\left(\frac{\beta +3\gamma}{4}+1+\frac{r}{2}\right)}.
  \endaligned
\end{equation}

\end{theorem}

\begin{proof}
 Similarly in the proof of Theorem \ref{new-thm-1}, we can establish the results here.
 Setting (i) $x=-1$ and $\alpha=  \frac{1}{2}\beta+ \frac{3}{2}\gamma+1-i $ $\left(i \in \mathbb{N}_0\right)$
  and (ii)$x=-1$ and $\alpha=  \frac{1}{2}\beta+ \frac{3}{2}\gamma+1+i $  $\left(i \in \mathbb{N}_0\right)$
  in \eqref{eq2-7} with the help of \eqref{G-Gauss-SST-Ra-Ra-1} and \eqref{G-Gauss-SST-Ra-Ra-2} yields,
  respectively,  \eqref{new-thm-15-eq1} and \eqref{new-thm-15-eq2}. We omit the details.
\end{proof}

\vskip 3mm

\begin{theorem}\label{new-thm-16}
  Let $i \in \mathbb{N}_0$. Then
  \begin{equation}\label{new-thm-16-eq1}
  \aligned
  &    F^{2:0;2}_{\,1:0;2} \left[\begin{array}{rrr}
                           1- \frac{1}{2}\beta+ \frac{1}{2}\gamma+i,\,\gamma   : &\overline{\hspace{3mm}} \,;\,& 1+\frac{1}{2}\gamma,\, \frac{\gamma-\beta}{2} \,; \\
                             \beta: &\overline{\hspace{3mm}} \,;\,& \frac{1}{2}\gamma ,\, 1+\frac{\gamma+\beta}{2} \,;
                             \end{array}-1,\,1
   \right] \\
  &  = \frac{2^{\frac{\gamma-\beta}{2} -\alpha+i}\,\Gamma\left(\frac{\beta-\gamma}{2}-i \right)\,
    \Gamma\left(1+\frac{\beta+\gamma}{2} \right) }{ \Gamma\left(\frac{\beta-\gamma}{2}\right)\,\Gamma\left(\beta -i \right)\,
        }\,
   \sum_{r=0}^{i}\,(-1)^r\,\binom{i}{r}\, \frac{\Gamma\left(\frac{\gamma+r+1}{2}\right)}
      {\Gamma\left(\frac{\beta+r+1}{2}-i\right)}
  \endaligned
\end{equation}
and
    \begin{equation}\label{new-thm-16-eq2}
 \aligned
   &   F^{2:0;2}_{\,1:0;2} \left[\begin{array}{rrr}
                           1- \frac{1}{2}\beta+ \frac{1}{2}\gamma -i,\,\gamma   : &\overline{\hspace{3mm}} \,;\,& 1+\frac{1}{2}\gamma,\, \frac{\gamma-\beta}{2} \,; \\
                             \beta: &\overline{\hspace{3mm}} \,;\,& \frac{1}{2}\gamma ,\, 1+\frac{\gamma+\beta}{2} \,;
                             \end{array}-1,\,1
   \right] \\
  &  = \frac{2^{\frac{\gamma-\beta}{2} -\alpha-i}\, \Gamma\left(1+\frac{\beta+\gamma}{2} \right) }{ \Gamma\left(\beta +i \right)
        }\,
   \sum_{r=0}^{i}\,\binom{i}{r}\, \frac{\Gamma\left(\frac{\gamma+r+1}{2}\right)}
      {\Gamma\left(\frac{\beta+r+1}{2}\right)}.
  \endaligned
\end{equation}

\end{theorem}

\begin{proof}
 Similarly in the proof of Theorem \ref{new-thm-1}, we can establish the results here.
 Setting (i) $x=-1$ and $\alpha= 1- \frac{1}{2}\beta+ \frac{1}{2}\gamma+i $ $\left(i \in \mathbb{N}_0\right)$
  and (ii)$x=-1$ and $\alpha= 1- \frac{1}{2}\beta+ \frac{1}{2}\gamma-i $  $\left(i \in \mathbb{N}_0\right)$
  in \eqref{eq2-7} with the help of \eqref{G-Bailey-ST-Ra-Ra-1} and \eqref{G-Bailey-ST-Ra-Ra-2} yields,
  respectively,  \eqref{new-thm-16-eq1} and \eqref{new-thm-16-eq2}. We omit the details.
\end{proof}

\section{Special cases and remarks}\label{SCR}

The particular cases $i=0$ in Eqs. \eqref{new-thm-1-eq1} or \eqref{new-thm-1-eq2}, \eqref{new-thm-2-eq1} or \eqref{new-thm-2-eq2},
\eqref{new-thm-3-eq1} or \eqref{new-thm-3-eq2}, the result in Theorem  \ref{new-thm-4}, Eqs. \eqref{new-thm-5-eq1} or \eqref{new-thm-5-eq2}, \eqref{new-thm-6-eq1} or \eqref{new-thm-6-eq2}, \eqref{new-thm-7-eq1} or \eqref{new-thm-7-eq2}, \eqref{new-thm-8-eq1} or \eqref{new-thm-8-eq2}, \eqref{new-thm-9-eq1} or \eqref{new-thm-9-eq2}, \eqref{new-thm-10-eq1} or \eqref{new-thm-10-eq2}, \eqref{new-thm-11-eq1} or \eqref{new-thm-11-eq2}, \eqref{new-thm-12-eq1} or \eqref{new-thm-12-eq2}, \eqref{new-thm-13-eq1} or \eqref{new-thm-13-eq2},  \eqref{new-thm-14-eq1} or \eqref{new-thm-14-eq2},  \eqref{new-thm-15-eq1} or \eqref{new-thm-15-eq2}, \eqref{new-thm-16-eq1} or  \eqref{new-thm-16-eq2},
  yield  known results, 
 respectively,  Corollaries 5.1 (2), 5.1 (3),  5.2 (1), 5.2 (2), 
   5.3 (a),  5.3 (b), 5.3 (c),  5.4 (a),    5.4 (b),  5.4 (c),  5.7 (b),  5.7 (c),  5.8,  5.9 (a),  5.9 (b), and  5.9 (c)
in Lin and Wang \cite{Li-Wa}.

\vskip 3mm 
Setting   $i=0,\,1,\,2,\,3,\,4,\,5$ in the results of  Theorems  \ref{new-thm-1}  to  \ref{new-thm-16}
gives those identities in the very recent paper \cite{Ch-Ra-MTJPAM}.
Yet the methods and other details in Theorems  \ref{new-thm-1}  to  \ref{new-thm-16} are seen 
 mainly to follow from those in \cite{Ch-Ra-MTJPAM}.

\bigskip
\renewcommand{\refname}{

\begin{table}[h]
\caption{Table for $\mathcal{A}_{i}$ and   $\mathcal{B}_{i}$}\label{table-Ai-Bi}
\begin{center}
\begin{tabular}{|r ||c|c|}\hline
$i$  & $\mathcal{A}_{i}$ & $\mathcal{B}_{i}$ \\
\hline\hline
  $-5$ &  $4(a-b-4)^2-2b(a-b-4)-b^2$  &  $4(a-b-4)^2+2b(a-b-4)-b^2$ \\
    {} &    $-8(a-b-4)-7b$            &  $+16(a-b-4)-b+12$  \\
\hline
 $-4$ &  $2(a-b-3)(a-b-1)-b(b+3)$  &  $4(a-b-2)$ \\
\hline
 $-3$ &  $2a-3b-4$  &  $2a-b-2$ \\
\hline
 $-2$ &  $a-b-1$  &  $2$ \\
\hline
 $-1$ &  $1$  &  $1$ \\
\hline
 $0$ &  $1$  &  $0$ \\
\hline
 $1$ &  $-1$  &  $1$ \\
\hline
 $2$ &  $1+a-b$  &  $-2$ \\
\hline
 $3$ &  $3b-2a-5$  &  $2a-b+1$ \\
\hline
 $4$ &  $2(a-b+3)(1+a-b)-(b-1)(b-4)$  &  $-4(a-b+2)$ \\
\hline
$5$ &  $-4(6+a-b)^2+2b(6+a-b)+b^2$  &  $4(6+a-b)^2+2b(6+a-b)-b^2$ \\
    {} &    $+22(6+a-b)-13b-22$            &  $-34(6+a-b)-b+62$  \\
\hline

\end{tabular}
\end{center}
\end{table}

\begin{table}[h]
\caption{Table for $\mathcal{C}_{i}$ and   $\mathcal{D}_{i}$}\label{table-Ci-Di}
\begin{center}
\begin{tabular}{|r ||c|c|}\hline
$i$  & $\mathcal{C}_{i}$ & $\mathcal{D}_{i}$ \\
\hline\hline
  $-5$ &  $(b+a-4)^2-\frac{1}{4}(b-a-4)^2$  &  $(b+a-4)^2-\frac{1}{4}(b-a-4)^2$ \\
    {} &    $-\frac{1}{2}(b+a-4)(b-a-4)$            & $+\frac{1}{2}(b+a-4)(b-a-4)$ \\
     {} &    $+4(b+a-4)-\frac{7}{2}(b-a-4)$            &  $+8(b+a-4)-\frac{1}{2}(b-a-4)+12$  \\

\hline
 $-4$ &  $\frac{1}{2}(b+a-3)(b+a+1)$  &  $2(b+a-1)$ \\
 {} &  $-\frac{1}{4}(b-a-3)(b-a+3)$  &  {} \\
\hline
 $-3$ &  $\frac{1}{2}(3a+b-2)$  &  $\frac{1}{2}(3b+a-2)$ \\
\hline
 $-2$ &  $\frac{1}{2}(b+a-1)$  &  $2$ \\
\hline
 $-1$ &  $1$  &  $1$ \\
\hline
 $0$ &  $1$  &  $0$ \\
\hline
 $1$ &  $-1$  &  $1$ \\
\hline
 $2$ &  $\frac{1}{2}(b+a-1)$  &  $-2$ \\
\hline
 $3$ &  $-\frac{1}{2}(3a+b-2)$  &  $\frac{1}{2}(a+3b-2)$ \\
\hline
 $4$ &  $\frac{1}{2}(b+a-3)(b+a+1)$  &  $2(b+a-1)$ \\
 {} &  $-\frac{1}{4}(b-a+3)(b-a-3)$  &  {} \\
\hline
$5$ &  $-(b+a+6)^2+\frac{1}{4}(b-a+6)^2$  &  $(b+a+6)^2-\frac{1}{4}(b-a+6)^2$ \\
    {} &    $+\frac{1}{2}(b-a+6)(b+a+6)$            & $+\frac{1}{2}(b+a+6)(b-a+6)$ \\
     {} &    $+11(b+a+6)-\frac{13}{2}(b-a+6)-20$            &  $-17(b+a+6)-\frac{1}{2}(b-a+6)+62$  \\

\hline

\end{tabular}
\end{center}
\end{table}

\begin{table}[h]
\caption{Table for $\mathcal{E}_{i}$ and   $\mathcal{F}_{i}$}\label{table-Ei-Fi}
\begin{center}
\begin{tabular}{|r ||c|c|}\hline
$i$  & $\mathcal{E}_{i}$ & $\mathcal{F}_{i}$ \\
\hline\hline
  $-5$ &  $4b^2-2ab-a^2+8b-7a$  &  $4b^2+2ab-a^2+16b-a+12$ \\
   \hline
 $-4$ &  $2b^2-a^2+4b-6a$  &  $4(b+1)$ \\
\hline
 $-3$ &  $2b-a$  &  $a+2b+2$ \\
\hline
 $-2$ &  $b$  &  $2$ \\
\hline
 $-1$ &  $1$  &  $1$ \\
\hline
 $0$ &  $1$  &  $0$ \\
\hline
 $1$ &  $-1$  &  $1$ \\
\hline
 $2$ &  $b-2$  &  $-2$ \\
\hline
 $3$ &  $a-2b-3$  &  $a+2b-7$ \\
\hline
 $4$ &  $2b^2-a^2-12b+5a+12$  &  $-4b+12$ \\
\hline
$5$ &  $-4b^2+2ab+a^2+22b-13a-20$  &  $4b^2+2ab-a^2-34b-a+62$ \\
  \hline

\end{tabular}
\end{center}
\end{table}

\end{document}